\documentclass[reqno,12pt]{amsart}
\usepackage{bbm}
\usepackage[all,pdf]{xy}
\usepackage{epsfig}
\usepackage{amsmath}
\usepackage{amssymb}
\usepackage{amscd}
\usepackage{graphicx}
\usepackage{color}

\usepackage{mathptmx}

\allowdisplaybreaks[4]

\makeatletter
\@namedef{subjclassname@2020}{%
	\textup{2020} Mathematics Subject Classification}
\makeatother
\usepackage[colorlinks=true]{hyperref}
\usepackage{amsfonts}

\usepackage[margin=0.8in]{geometry}

\newtheorem{thm}{Theorem}[section]
\newtheorem{lemma}[thm]{Lemma}

\newtheorem{ques}[thm]{Question}

\newtheorem{defn}[thm]{Definition}
\newtheorem{rem}[thm]{Remark}

\def \N {\mathbb N}
\def \C {\mathbb C}
\def \Z {\mathbb Z}

\def\FS {{\bf FS}}
\def\FP {{\bf FP}}

\def\AP {{\bf AP}}
\def\GP {{\bf GP}}
\numberwithin{equation}{section}

\begin{document}

\title[]{Monochromatic sums and products with additive or multiplicative shifts in natural numbers}

	\author[]{Wen Huang, Song Shao, Tianyi Tao, Rongzhong Xiao, and Ningyuan Yang}	
	\address[Wen Huang]{School of Mathematical Sciences, University of Science and Technology of China, Hefei, Anhui, 230026, PR China}
	\email{wenh@mail.ustc.edu.cn}
	
	\address[Song Shao]{School of Mathematical Sciences, University of Science and Technology of China, Hefei, Anhui, 230026, PR China}
	\email{songshao@ustc.edu.cn}
	
	\address[Tianyi Tao]{School of Mathematical Sciences, Fudan University, Shanghai, 200433, PR China}
	\email{tytao20@fudan.edu.cn}
	
	\address[Rongzhong Xiao]{School of Mathematical Sciences, University of Science and Technology of China, Hefei, Anhui, 230026, PR China \& Aix-Marseille Université, CNRS, Institut de Mathématiques de Marseille, Marseille, France}
	\email{xiaorz@mail.ustc.edu.cn}
	
	\address[Ningyuan Yang]{School of Mathematical Sciences, Fudan University, Shanghai, 2004-33, PR China}
	\email{nyyang23@m.fudan.edu.cn}
	
	\subjclass[2020]{Primary: 05D10; Secondary: 37B20, 03E05.}
	\keywords{Central sets, Finite sums, Idempotents, Topological dynamical system, Ultrafilters.}

\begin{abstract}

In this paper, we prove that for any finite coloring of $\N$, there are $\lambda,\rho\in \N$ such that there exist (infinitely many) pairs $(x,y),(u,v)\in \N^2$ such that the two sets $\{\lambda x,\lambda y,xy,\lambda(x+y)\}$ and $\{u+\rho,v+\rho,uv+\rho,u+v\}$ are monochromatic. By further employing related arguments, we can provide two different proofs for a special case of Milliken–Taylor theorem.

\end{abstract}

\maketitle
\section{Introduction}
\subsection{Background and motivation}
Throughout the paper, all semigroups, mentioned by us, are discrete and infinite.  
Now, let us recall the definition of the finite sums.

\begin{defn}
Let $(S,+)$ be a semigroup, and $\{s_i\}_{i\in \Z}\subseteq S$.
\begin{itemize}
\item For a non-empty finite subset $\alpha=\{\alpha_1<\alpha_2<\dots<\alpha_k\}\subseteq\Z$, let
\[
s_\alpha=s_{\alpha_1}+s_{\alpha_2}+\dots+s_{\alpha_k}.
\]
\item For a non-empty set $I\subseteq\Z$, denote the finite sum of $\{s_i\}_{i\in I}$ by
\[
\FS_{(S,+)}(\{s_i\}_{i\in I})=\{s_\alpha:\ \alpha\ \text{is a non-empty finite subset of}\ I\}.
\]
\end{itemize}
When the index set $I=\{a_n<\cdots<a_m\}$, we write $\FS_{(S,+)}(s_{a_n},\ldots,s_{a_m})$ for $\FS_{(S,+)}(\{s_i\}_{i\in I})$ at some cases. 
\end{defn}

In 1974, Hindman proved the following result, which is known as the Finite Sum Theorem or Hindman's theorem at present. 

\begin{thm}\label{N0}
$($\cite[Theorem 3.1]{H74}$)$ For any finite coloring of $\N$, there is a sequence $\{a_i\}_{i=0}^{\infty}$ in $\N$ such that $\FS_{(\N,+)}(\{a_i\}_{i=0}^{\infty})$ is monochromatic.
\end{thm}

After Hindman's finite sum theorem, Furstenberg \cite[Definition 8.3]{F} introduced the notion of central sets for the semigroup $(\N,+)$ by using dynamical language. Based on this, Furstenberg provided a dynamical proof for Theorem \ref{N0}. Later, Bergelson and Hindman \cite{BH90} generalized Furstenberg's proof to general semigroups. In fact, one can also find an ultrafilter (For its definition, see Section \ref{subsec2-1}.) proof from Glazer (see \cite[Section 10]{C77}). In \cite[First paragraph of Section 3, Chapter 8]{F}, Furstenberg wrote the following sentence to illustrate the relation between two proofs:

 \emph{``Glazer's proof of Hindman's theorem runs along very similar lines to the present proof (using dynamical approach) but without the dynamical appurtenances.''}

At present, by the partition regularity of central sets, we know that Hindman's theorem can be deduced from the following. 

\begin{thm}\label{N1}
$($\cite[Theorem 16.4]{HS}$)$ Let $(S,+)$ be a semigroup and $A\subseteq S$ be central (see Definition \ref{defn}) in $(S,+)$. Then there is a sequence $\{a_i\}_{i=0}^{\infty}$ in $S$ such that $$\FS_{(S,+)} (\{a_i\}_{i=0}^{\infty})\subseteq A.$$
\end{thm}

Since there are two semigroup operations $+$ and $\cdot$ on $\N$, one naturally seeks to search for those structures containing additive and multiplicative operation simultaneously. For instance, Bergelson and Hindman built the following.
\begin{thm}\label{special}
$($\cite[Corollary 5.5 and Theorem 5.6]{BH90}$)$ For any finite coloring of $\N$, there is some monochromatic subset such that it is both central in $(\N,+)$ and $(\N,\cdot)$. Moreover, there is some ultrafilter $p$ such that for all $C\in p$, $C$ is both central in $(\N,+)$ and $(\N,\cdot)$.
\end{thm}
By Theorem \ref{special} and Theorem \ref{N1}, one can prove that for any finite coloring of $\N$, there are two sequences $\{x_i\}_{i\ge 0}$ and $\{y_i\}_{i\ge 0}$ in $\N$ such that the set $$\FS_{(\N,+)}(\{x_i\}_{i\ge 0})\cup \FS_{(\N,\cdot)}(\{y_i\}_{i\ge 0})$$ is monochromatic.
Hence, it is natural to ask the following question:

\emph{Is it true that for any finite coloring of $\N$, there is a sequence $\{x_i\}_{i\ge 0}$ in $\N$ such that the set $$\FS_{(\N,+)}(\{x_i\}_{i\ge 0})\cup \FS_{(\N,\cdot)}(\{x_i\}_{i\ge 0})$$ is monochromatic ?}

 Unfortunately, it fails. One can see \cite[Theorem 2.15]{H80} or \cite[Theorem 5]{HML} for related counterexamples.
However, for a sequence with finite length rather than infinite length, the case is unknown. Clearly, the simplest case is the following question, which was asked by Hindman on numerous occasions. 

\begin{ques}\label{q1}
$($\cite[Question 3]{HLS}$)$ Is it true that for any finite coloring of $\N$, there exist $x,y\in \N$ such that $\{x,y,xy,x+y\}$ is monochromatic ? 
\end{ques}
The question was first studied by Graham (see \cite[Theorem 4.3]{H79}), who proved that for any $2$-coloring of $\{1,\ldots,252\}$, one can always find monochromatic $\{x,y,xy,x+y\}$, and by Hindman, who showed that for any $2$-coloring of $\{2,\ldots,990\}$, such monochromatic pattern always exists. In 2022, Bowen \cite{B22} proved that for any $2$-coloring of $\N$ and any $k\in \N$, there are monochromatic sets of the form $\{x,y,xy,x+ky\}$. 

For general case, there are only a few results. In Moreira's breakthrough paper \cite{M17}, he proved that for any finite coloring of $\N$, there exist $x,y\in \N$ such that $\{x,xy,x+y\}$ is monochromatic. Later, for this result, Alweiss gave a new proof in \cite{A22}.

Likely, one can consider Question \ref{q1} on some semirings and (finite) fields. For related results, one can see \cite{S10,GS,BM17,BS22, X24,A23,TY,K24}.
\subsection{Two results on Question \ref{q1}}
Next, two results on similar forms as $\{x,y,xy,x+y\}$ will be provided. Before stating them, we introduce some notions. We say that a subset $A$ of $\N$ contains $k$-\AP\ (The notation \AP\ comes from the phrase ``arithmetic progression".) if there is an arithmetic progression $\{a,a+d,\dots,a+(k-1)d\}\subseteq A$ and it contains $k$-\FS\ if $\FS_{(\N,+)}(\{a_i\}_{i=1}^{k})\subseteq A$. Likely, one can give the definitions for $k$-\GP\ (The notation \GP\ comes from the phrase ``geometric progression".) and $k$-\FP\ (The notation \FP\ comes from the phrase ``finite product".) when the addition is replaced by the multiplication.

\begin{thm}\label{simple}
For any finite coloring of $\mathbb N$, there exists some $\lambda\in\mathbb N$ such that for any $k\in\mathbb N$, there exist a finite subet $A$ of $\N$ containing $k$-\AP\ and $k$-\FS, and a central subset $B$ in $(\N,+)$ such that
\[\lambda A\cup\lambda B\cup\lambda(A+B)\cup AB\]
is monochromatic.
\end{thm}
Immediately, from the above theorem, we know that for any finite coloring of $\N$, there is $a\in \N$ such that there are (infinitely many) pairs $(x,y)\in \N^2$ such that the set $\{ax,ay,xy,a(x+y)\}$ is monochromatic.
\begin{thm}\label{simple2}
For any finite coloring of $\mathbb N$, there exists some $\rho\in\mathbb N$ such that for any $k\in\mathbb N$, there exist a finite subset $A$ of $\N$ containing $k$-\AP, $k$-\GP, $k$-\FS\ and $k$-\FP, and a subset $B$ which is both central in $(\N,+)$ and $(\N,\cdot)$ such that
\[(\rho+A)\cup(\rho+B)\cup(\rho+AB)\cup (A+B)\]
is monochromatic.
\end{thm}
Clearly, based on the above theorem, we know that for any finite coloring of $\N$, there is $b\in \N$ such that there are (infinitely many) pairs $(u,v)\in \N^2$ such that the set $\{u+b,v+b,uv+b,u+v\}$ is monochromatic.

\subsection{Extensions of Finite Sum Theorem}

Now, we return to the Finite Sum Theorem. Firstly, we rewrite it as the following:

For any finite coloring of $\N$, there is a sequence $\{a_i\}_{i=0}^{\infty}$ in $\N$ such that the set $$ \bigcup_{0\le m\le n}\{x:x\in\FS_{(\N,+)}(\{a_j\}_{j=m}^{n})\}$$ is monochromatic.

Milliken \cite{1975Ramsey} and Taylor \cite{1976A} provided its higher-dimensional generalization as follows:
\begin{thm}\label{higher}
$($\cite{1975Ramsey,1976A}$)$ Let $(S,+)$ be a semigroup. For any $d\in \N$ and any finite coloring of $S^d$, there is a sequence $\{a_i\}_{i=0}^{\infty}$ in $S$ such that the set
\[
\bigcup_{(m_0,\ldots,m_d)\in \Z^{d+1},\atop 0\leq m_0<m_1<\dots<m_d}\{(x_1,\dots,x_d):\ x_i\in\FS_{(S,+)}(\{a_j\}_{j=m_{i-1}}^{m_{i}-1})\ \text{for each}\ 1\leq i\leq d\}
\]
is monochromatic.
\end{thm}

Although Theorem \ref{higher} is known, we will employ our technique to provide two different proofs for it to exhibit the essence of our technique and extract the following Theorem \ref{multi} and \ref{multi2} in this paper.

\begin{thm}\label{multi}
Let $(S,+)$ be a semigroup and $\circ:S^2\to S$ be a binary operation. For any finite coloring of $S$ and any $D\in\N$, there is a sequence $\{a_i\}_{i=0}^{\infty}$ in $S$ such that for any $1\le d\le D$, the set
\[
\bigcup_{(m_0,\ldots,m_d)\in \Z^{d+1},\atop 0\leq m_0<m_1<\dots<m_d}\{(((x_1\circ x_2)\circ x_3)\circ\cdots)\circ x_d:\ x_i\in\FS_{(S,+)}(\{a_j\}_{j=m_{i-1}}^{m_{i}-1})\ \text{for each}\ 1\leq i\leq d\}
\]
is monochromatic, where the color of the set depends on $d$.
\end{thm}
If the cases of $d=1$ and $d=2$ of Theorem \ref{multi} correspond to the same color, then Question \ref{q1} has a positive answer.

Next, we will illustrate Theorem \ref{multi} for the semigroup $(\N,+)$ and multiplication operation to let the readers understand what happens. In this case, by Theorem \ref{multi}, one knows that for any finite coloring of $\N$ and any $D\in\N$, there is a sequence $\{a_i\}_{i=0}^{\infty}$ in $\N$ such that for any $1\le d\le D$, the set $$\bigcup_{(m_0,\ldots,m_d)\in \Z^{d+1}, \atop 0\leq m_0<m_1<\dots<m_d}\prod_{i=1}^d\FS_{(\N,+)}(\{a_j\}_{j=m_{i-1}}^{m_{i}-1})$$ is monochromatic, where the color of the set depends on $d$. When $d=D=1$, it is the Finite Sum Theorem. Moreover, it also implies that for any finite coloring of $\N$, there are $a,x,y\in \N$ such that the set $\{ax,ay,xy,a(x+y)\}$ is monochromatic. Note that Theorem \ref{simple2} cannot be derived directly. Its proof relies on a special multiplicative minimal idempotent ultrafilter, which stems from the deep interplay between multiplication and addition within ultrafilter theory.

Specially, when $(S,+)=(\N,+)$, we have the following result, which is similar to Theorem \ref{multi}. 
\begin{thm}\label{multi2}
For any binary operation $\circ:\N^2\to \N$ and any finite coloring of $\N$, there is a sequence $\{a_i\}_{i=0}^{\infty}$ in $\N$ such that the set
\begin{align*}
\bigcup_{m_0\in\N}\bigcup_{d\in\FS_{(\N,+)}(\{a_j\}_{j=0}^{m_{0}-1})}\bigcup_{(m_1,\ldots,m_d)\in \N^{d},\atop m_0<m_1<\dots<m_d}&\{(((x_1\circ x_2)\circ x_3)\circ\cdots)\circ x_d:\\
&x_i\in\FS_{(\N,+)}(\{a_j\}_{j=m_{i-1}}^{m_{i}-1})\ \text{for each}\  1\leq i\leq d\}
\end{align*}
is monochromatic.
\end{thm}

\subsection*{Structure of the paper} In Section \ref{sec2}, we give some preliminaries. In Section \ref{sec2.5}, we prove Theorem \ref{simple} and Theorem \ref{simple2}. In Section \ref{sec3}, \ref{sec4}, we provide a dynamical proof and an ultrafilter proof for Theorem \ref{higher} and illustrate the relation between two proofs. In Section \ref{sec5}, we prove Theorem \ref{multi} and Theorem \ref{multi2}. 

\subsection*{Acknowledgement} 
 The first author is supported by National Natural Science Foundation of China (12090012, 12031019, 12090010). The second author is supported by National Natural Science Foundation of China (12371196). The fourth author is supported by National Natural Science Foundation of China (123B2007, 12371196).

\section{Preliminaries}\label{sec2}

In this section, we recall some notions and results.
\subsection{Notations}
Let $\N(\text{resp.}\ \Z,\C)$ denote the natural numbers (resp. integers, comlex numbers) set. Let $\N_0=\{n\in \Z:n\ge 0\}$.

For any semigroup $(S,+)$, $s\in S$ and non-empty subset $A$ of $S$, let $A-s=\{t\in S:s+t\in A\}$. 

For any two non-empty subsets $A$ and $B$ of $\N$ and any $x\in \N$, let $A+B=\{a+b:a\in A,b\in B\},AB=\{ab:a\in A,b\in B\},x+A=\{x+a:a\in A\}$ and $xA=\{xa:a\in A\}$.

For any non-empty set $X$ and any $d\in \N$, let $X^d=\{(x_1,\ldots,x_d):x_i\in X,1\le i\le d\}$.
\subsection{$\beta S$ and central sets}\label{subsec2-1}
Let $S$ be a discrete, infinite topological space. An \textbf{ultrafilter} on $S$ is a non-empty family $p$ of subsets of $S$ sufficing the following:
\begin{itemize}
\item If $A,B\in p$, then $A\cap B\in p$;
\item If $A\in p$ and $A\subset B$, then $B\in p$;
\item For any $E\subset S$, $E\in p\iff E^c\notin p$.
\end{itemize}

Let $\beta S$ be the set of all ultrafilters on $S$. Let $\mathcal{F}_S=\{\hat{E}:=\{p\in \beta S:E\in p\}:E\subseteq S,E\neq \varnothing\}$. By \cite[Lemma 3.17]{HS}, $\mathcal{F}_S$ is a topological basis. By \cite[Theorem 3.18.(a) and Theorem 3.27]{HS}, under such topology deriving from $\mathcal{F}_S$, $\beta S$ is a compact Hausdorff space. And it and the Stone-\v{C}ech compactification of $S$ are homeomorphic. 

If $(S,+)$ is a semigroup, the operation $+$ $:S^2\to S$ can be extended to $\beta S\times \beta S\to\beta S$ (see \cite[Theorem 4.1]{HS}). The resulting operation, which we still use the notion $+$ to denote, will satisfy the following (see \cite[Theorem 4.12.(b)]{HS}):
\[
p+q=\{A\subseteq S:\ \{s\in S:\ \{t\in S:\ s+t\in A\}\in q\}\in p\}.
\]
That is, $A\in p+q$ if and only if $\{s\in S:\ A-s\in q\}\in p$.

By \cite[Theorem 4.1.(b) and Theorem 4.4]{HS}, $(\beta S,+)$ is a compact Hausdorff right topological semigroup. (For more on $(\beta S,+)$, see \cite{HS} or \cite[Chapter 1]{NGL}.)

Next, we give some definitions.
\begin{defn}
	\begin{itemize}
		\item[(1)] $L$ is a left ideal of $(\beta S,+)$ if $\varnothing\neq L\subseteq \beta S$ and $\beta S + L\subseteq L$. 
		\item[(2)] $L$ is a minimal left ideal of $(\beta S,+)$ if $L$ is a left ideal of $(\beta S,+)$ and whenever $J$ is a left ideal of $(\beta S,+)$ and $J\subseteq L$, one has $J=L$. 
		\item[(3)] $p$ is an idempotent in $(\beta S,+)$ if $p\in \beta S$ and $p+p=p$.
		\item[(4)] Let $p_1,p_2$ be two idempotents in $(\beta S,+)$. We say that $p_1\leq_{L}p_2$ if $p_1=p_1+p_2$. 
		\item[(5)] $p$ is a minimal idempotent in $(\beta S,+)$ if $p$ is an idempotent in $(\beta S,+)$ and $p$ is minimal with respect to the order $\leq_L$.
	\end{itemize}
\end{defn}
By \cite[Theorem 1.38.(d) and Corollary 2.6]{HS}, we know that in $(\beta S,+)$, the subset of all minimal idempotents is non-empty. Now, we give the definition of central sets.
\begin{defn}\label{defn}
	A subset $A$ of $S$ is \textbf{central} in $(S,+)$ if $A$ is a member of some minimal idempotent in $(\beta S,+)$ .
\end{defn}
\subsection{Topological dynamical systems}\label{subsec2-2}
We say that a pair $(X,(T_s)_{s\in S})$ is a \textbf{topological dynamical system} if $X$ is a compact metric space, $(S,+)$ is a semigroup, $T_s:X\to X$ is continuous for all $s\in S$ and $T_s\circ T_t=T_{s+t}$ for all $s,t\in S$.

Given a topological dynamical system $(X,(T_s)_{s\in S})$, we define $\varphi:S\to X^{X}$ by $s\mapsto (T_sx)_{x\in X}$, where $X^{X}$ is equipped with the product topology. Since $\beta S$ coincides with the Stone-\v{C}ech compactification of $S$, $\beta S$ can act on $X$ continuously by lifting $\varphi$ from $S$ to $\beta S$.  Let $px$ be the value of the $x$-th coordinate of $\varphi(p)$ for any $p\in\beta S,x\in X$. 

Next, we introduce a simple known lemma.
\begin{lemma}\label{prox}
Let $(X,(T_s)_{s\in S})$ be a topological dynamical system, $x\in X$ and $p\in\beta S$. Then for all non-empty neighborhood $U$ of $px$, 
\[
E:=\{s\in S:\ T_sx\in U\}\in p.
\]
\end{lemma}
If we require $p$ to be a minimal idempotent in $(\beta S,+)$, then the subset $\{s\in S:\ T_s(x)\in U\}$ is central in $(S,+)$ since it is in $p$. Moreover, since $p+p=p$, $\{s\in S:\ T_s(px)\in U\}\in p$.

\section{Proofs of Theorem \ref{simple} and \ref{simple2}}\label{sec2.5}

In this section, we prove Theorem \ref{simple} and Theorem \ref{simple2}, which will display how our method works.  

\begin{proof}[Proof of Theorem \ref{simple}]
Fix $c\in \N$ with $c>1$. Let $X:=\{1,\dots,c\}^{\N_0\times\N_0}$ be equipped with the product topology. We define the shift transformation $\sigma_s$ for each $s\in\N$ by
\[
(\sigma_sx)(m,n)=x(m,n+sm)\ \text{for each}\ x\in X,m,n\in \N_0.
\]
Then $(X,(\sigma_s)_{s\in \N})$ is a topological dynamical system.

Fix a $c$-coloring $\xi_0\in\{1,\dots,c\}^{\N_0}$. We define $\xi\in X$ by
\[
\xi(m,n)=\xi_0(n)\ \text{for each}\ m,n\in\N_0.
\]
Take a minimal idempotent $p$ from $(\beta \N,+)$ arbitrarily. Note that $p(\xi,p\xi)=(p\xi,p\xi)$. Then by Lemma \ref{prox}, for any non-empty neighborhood $V$ of $p\xi$, one has
\[
\{s\in\N:\ \sigma_s\xi\in V,\sigma_s(p\xi)\in V\}\in p.
\]
For each $k\in\N_0$, let
\[
U_k=\{x\in X:\ x(m,n)=(p\xi)(m,n)\ \text{for any}\ m,n\leq k\}
\]
and
\[
E_k=\{s\in\N:\ \sigma_s\xi\in U_k,\sigma_s(p\xi)\in U_k\}\in p.
\] 
And we let $\xi_1\in\{1,\dots,c\}^{\N_0}$ such that
\[
\xi_1(m)=(p\xi)(m,0)\ \text{for each}\ m\in \N_0.
\]

When we restrict the map $\xi_1:\N_0\to \{1,\ldots,c\}$ to $\N$, it gives a finite coloring of $\N$. Clearly, for this coloring of $\N$, there is a color class $C\in p$ with color $i_0$, where $i_0\in \{1,\ldots,c\}$. 

Take $\lambda$ from $E_0\cap C$ arbitrarily. Then for all $a\in E_{\lambda}\cap C$, one has
\begin{equation}\label{eq3-1}
	\xi_0(\lambda a)=\xi(\lambda,\lambda a)=(\sigma_{a}\xi)(\lambda,0)=(p\xi)(\lambda,0)=\xi_1(\lambda)=i_0.
\end{equation}
Note that $E_\lambda\cap C\in p$. By \cite[Lemma 5.19 and Theorem 14.1]{HS}, we know that for any $k\in\N$, one can find a non-empty finite set $A\subseteq E_\lambda\cap C$ containing $k$-\AP\ and $k$-\FS.
By \eqref{eq3-1}, $\lambda A$ has color $i_0$.

Let $M={\rm max}(\{\lambda\}\cup A)$. Likely, we know that for each non-empty set $D\subseteq E_M\cap C$, $\lambda D$ and $AD$ are both of color $i_0$. 

Since $\sigma_s$ is continuous for each $s\in \N$, there is an $N\in\N$ with $N>M$ such that
\[
\sigma_a(U_N)\subseteq U_\lambda\ \text{for each}\ a\in A.
\]
Let $B=E_N\cap C\subset E_M\cap C$. Then for each $a\in A$ and each $b\in B$, there is
\begin{align*}
	& \xi_0(\lambda(a+b))=\xi(\lambda,\lambda(a+b))=(\sigma_{a+b}\xi)(\lambda,0)
	\\  & \hspace{3cm} =
	(\sigma_a(\sigma_b\xi))(\lambda,0)=(p\xi)(\lambda,0)=\xi_1(\lambda)=i_0.
\end{align*}
So, $\lambda(A+B)$ is of color $i_0$. 

Note that $B\in p$. So, $B$ is a central subset in $(\N,+)$.

Now, we have proved that for any finite coloring of $\N_0$, there is some $\lambda\in\N$ such that for any $k\in \N$, there are a non-empty finite set $A\subseteq \N$ containing $k$-\AP\ and $k$-\FS\ and a central set $B$ in $(\N,+)$ such that $$\lambda A\cup\lambda B\cup\lambda(A+B)\cup AB$$ is monochromatic.

Note that every finite coloring of $\N$ can be extended to be a finite coloring of $\N_0$ and $\{\lambda\}\cup A\cup B$ is a subset of $\N$.
Then this completes the proof.
\end{proof}
\begin{rem}
	In fact, the method used in the above proof can prove the following:
	
	Let $(S,+)$ be a semigroup and $\{f_s\}_{s\in S}$ be a sequence of self-homomorphisms on $(S,+)$. Then for any finite coloring of $S$ and any $r\in \N$, there exists some $\lambda\in S$ such that there are $r$-elements set $A\subseteq S$, infinite set $B\subseteq S$ which is central in $(S,+)$ and $z_1,\ldots,z_r\in S$ such that $$f_{\lambda}(A)\cup f_{\lambda}(B)\cup f_{\lambda}(A+B)\cup \{f_{x}(y):x\in A,y\in B\}$$ is monochromatic and 
	$$f_{\lambda}(\FS_{(S,+)}(z_1,\ldots,z_r))\cup\{f_{z_i}(z_j):1\le i<j\le r\}$$ is monochromatic.
\end{rem}
Next, we prove Theorem \ref{simple2}.
\begin{proof}[Proof of Theorem \ref{simple2}]
	By \cite[Theorem 5.6]{BH90}, there is a minimal idempotent $p$ in $(\beta \N,\cdot)$ such that for any $C\in p$, $C$ is central both in $(\N,+)$ and $(\N,\cdot)$. For any finite coloring of $\N$, there is a color class $D\in p + p$.
	
	Choose $\lambda$ from $\N$ such that $D':=D-\lambda \in p$. Put $$E=\{n\in \N:D'/n\in p\}\cap \{n\in \N:D-n\in p\}\cap D'.$$ Clearly, $E\in p$. By \cite[Theorem 14.1 and Theorem 16.4]{HS}, we know that one can find a non-empty finite set $A\subseteq E$ containing $k$-\AP, $k$-\FS, $k$-\GP\ and $k$-\FP. Let $$B=D'\bigcap \left(\bigcap_{a\in A}D'/a\right)\bigcap \left(\bigcap_{a\in A}(D-a)\right).$$ Clearly, $B\in p$.
	
	Then $\lambda + B\subseteq \lambda + D' \subseteq D$, $A+B\subseteq D$, $\lambda + AB\subseteq \lambda + D' \subseteq D$, and $\lambda + A\subseteq \lambda + D' \subseteq D$. This finishes the proof.
\end{proof}
Actually, by the argument like the above proof, one may give an ultrafilter proof of Theorem \ref{simple}.

\section{The first proof of Theorem \ref{higher}: Dynamical method}\label{sec3}

In this section, we give a proof for Theorem \ref{higher} via the method from dynamical systems. The proof presented below is similar to the proof of Theorem \ref{simple}. But it is more index-heavy. 

From now on, we use the symbol $c$ to denote the number of colors of some finite coloring of an infinite set $D$ and $c>1$. If we let
\[
X^{(d)}:=\{1,\dots,c\}^{D^d}
\]
be equipped with the product topology, then a $c$-coloring of $D^d$ can be viewed as an element in $X^{(d)}$.
To be convenient, we use the notation
$
(A_1,A_2,\dots,A_d)
$
to represent the set
$
\{(x_1,\dots,x_d):\ x_i\in A_i,1\leq i\leq d\}, 
$
where $A_i\subseteq D,\varnothing\neq A_i,1\le i\le d$. Similarly, if $f$ is a function mapping $D^d$ to $\C$, let
$
f(A_1,A_2,\dots,A_d)=\{f(x_1,\dots,x_d):\ x_i\in A_i,1\leq i\leq d\}.
$

\begin{proof}[Dynamical proof of Theorem \ref{higher}]
Without loss of generality, we assume that $S$ has no two-side identical elements. Now, we add an element $0$ to $S$, which satisfies the property that for any $s\in S$, $0+s=s+0=s$. Let $S_0=S\cup \{0\}$. For each $1\le k\leq d$, let
\[
X^{(k)}:=\{1,\dots,c\}^{S_0^k}
\]
be equipped with the product topology, and define the shift transformation $\sigma^{(k)}_s$ for each $s\in S$ by
\[
\left(\sigma^{(k)}_sx \right)(t_1,t_2,\dots,t_k)=x(t_1,t_2,\dots,t_k+s)\quad\text{for each}\  x\in X^{(k)},t_1,\dots,t_k\in S.
\]
Then for each $1\le k\le d$, $\left(X^{(k)},(\sigma^{(k)}_s)_{s\in S}\right)$ is a topological dynamical system.

For any $1\le k\le d$ and any non-empty $A\subseteq S,U\subseteq X^{(k)}$, let $$\sigma_{A}^{(k)}(U)=\bigcup_{s\in A}\sigma_{s}^{(k)}(U).$$

Fix $\xi^{(d)}\in X^{(d)}$. Next, we define an element
\[
\left(\xi^{(d)},\xi^{(d-1)},\dots,\xi^{(1)}\right)
\] in $X^{(d)}\times X^{(d-1)}\times\dots\times X^{(1)}$.
Take an idempotent $p$ from $(\beta S,+)$ arbitrarily and fix it. If $d=1$, we have no work to do. If $d>1$, assume that we have defined  $\xi^{(k+1)}(1\le k\le d-1)$. We define $\xi^{(k)}$ by
\[
\xi^{(k)}(t_1,\dots,t_k)=\left(p\xi^{(k+1)}\right)(t_1,\dots,t_k,0)\quad\text{for each}\  t_1,\dots,t_k\in S_0.
\]
For any $1\le k\le d$ and any non-empty finite subset $N^{(k)}$ of $S_{0}^{k}$, let
\[
U_{N^{(k)}}=\left\{x\in X^{(k)}:\ x(t_1,\dots,t_k)=\left(p\xi^{(k)}\right)(t_1,\dots,t_k)\ \text{for each}\  (t_1,\dots,t_k)\in N^{(k)}\right\}
\]
and
\[
E_{N^{(k)}}=\left\{s\in S:\ \sigma_s^{(k)}\xi^{(k)}\in U_{N^{(k)}},\sigma_s^{(k)}\left(p\xi^{(k)}\right)\in U_{N^{(k)}}\right\}.
\] Note that for any $1\le k\le d$, $p\left(\xi^{(k)},p\xi^{(k)}\right)=\left(p\xi^{(k)},p\xi^{(k)}\right)$. Then by Lemma \ref{prox}, one knows that for any $1\le k\le d$ and any non-empty finite subset $N^{(k)}$ of $S_{0}^{k}$, $E_{N^{(k)}}\in p$.

Next, we construct inductively
\begin{itemize}
\item a sequence $a_0,a_1,\dots$ in $S$
\item a sequence ${N^{(1)}_0},{N^{(1)}_1},\dots$ that are non-empty finite subsets of $S_0$
\item \dots
\item a sequence ${N^{(d)}_0},{N^{(d)}_1},\dots$ that are non-empty finite subsets of $S_{0}^d$
\end{itemize}
such that the following hold:
\begin{itemize}
\item[(C1)] $\text{for each}\  n\geq0$,
\[
a_n\in E_{N^{(1)}_n}\cap E_{N^{(2)}_n}\cap\dots\cap E_{N^{(d)}_n};
\]
\item[(C2)] $\text{for each}\  n\geq1,1\leq k\leq d$,
\[
\sigma^{(k)}_{a_{n-1}}\left(U_{N^{(k)}_n}\right)\subseteq U_{N^{(k)}_{n-1}}, N^{(k)}_n\supseteq(\FS_{(S,+)}(a_0,\dots,a_{n-1})\cup\{0\})^k\bigcup\left(\bigcup_{i=0}^{n-1}N_i^{(k)}\right);
\]
\item[(C3)] $\text{for each}\ n\ge 1,(m_0,\ldots,m_{k-1})\in \N_{0}^{k}\ \text{with}\  m_0<m_1<\dots<m_{k-1}<n,1\le k\leq \min(n,d)$, 
\begin{equation*}
\begin{split}
& \sigma^{(k)}_{\FS_{(S,+)}(a_{m_{k-1}},\dots,a_{n-1})}\left(U_{N^{(k)}_n}\right)\subseteq 
\\ & \hspace{2.5cm}
\begin{cases}
U_{\{0\}}, k=1;\\
U_{(\FS_{(S,+)}(a_{m_0},\dots,a_{m_1-1}),\dots,\FS_{(S,+)}(a_{m_{k-2}},\dots,a_{m_{k-1}-1}),\{0\})}, k>1.
\end{cases}
\end{split}
\end{equation*}
\end{itemize}

For each $1\le k\le d$, let $N^{(k)}_0=\{(0,0,\dots,0)\}\subseteq S_{0}^{k}$. Take $a_0$ from $E_{N^{(1)}_0}\cap E_{N^{(2)}_0}\cap\dots\cap E_{N^{(d)}_0}$ arbitrarily. Then for $n=0$, (C1) holds.

Suppose that we have constructed $\{a_j\}_{0\le j<n}$, $\{N_{j}^{(i)}\}_{0\le j<n,1\le i\le d}$ such that for any $0\le j<n$, (C1) - (C3) hold, where $n\ge 1$. Next, we construct $a_n,N_{n}^{(i)},1\le i\le d$ under the requirements of (C1) - (C3).

For any $1\leq k\leq d$, by applying (C1) to $n-1$, we have that
\[
\sigma^{(k)}_{a_{n-1}}\left(p\xi^{(k)}\right)\in U_{N^{(k)}_{n-1}}.
\]
Since $\sigma_{a_{n-1}}^{(k)}$ is continuous, there is a non-empty finite subset $N^{(k)}_n$ of $S_{0}^{k}$ such that
\[
\sigma^{(k)}_{a_{n-1}}\left(U_{N^{(k)}_n}\right)\subseteq U_{N^{(k)}_{n-1}}\ \text{and}\ N^{(k)}_n\supseteq(\FS_{(S,+)}(a_0,\dots,a_{n-1})\cup\{0\})^k\bigcup\left(\bigcup_{i=0}^{n-1}N_i^{(k)}\right).
\]
Take $a_n$ from $E_{N^{(1)}_n}\cap E_{N^{(2)}_n}\cap\dots\cap E_{N^{(d)}_n}$ arbitrarily. Then for $a_n$, (C1) holds and for $N_{n}^{(i)},1\le i\le d$, (C2) holds. 

Next, we verify that $N_{n}^{(i)},1\le i\le \min(d,n)$ satisfy (C3).

Fix $1\le k\le \min(d,n)$. Fix $(m_0,\ldots,m_{k-1})\in \Z^{k}\ \text{with}\ 0\leq m_0<m_1<\dots<m_{k-1}<n$. Let us consider the case $k>1$ first. If $n=m_{k-1}+1$, by applying (C2) to $n-1$, we know that 
\begin{align*}
\sigma^{(k)}_{\FS_{(S,+)}(a_{m_{k-1}},\dots,a_{n-1})}\left(U_{N^{(k)}_n}\right)&=\sigma^{(k)}_{a_{n-1}}\left(U_{N^{(k)}_n}\right)\subseteq U_{N^{(k)}_{n-1}}\\
&\subseteq U_{(\FS_{(S,+)}(a_{m_0},\dots,a_{m_1-1}),\dots,\FS_{(S,+)}(a_{m_{k-2}},\dots,a_{m_{k-1}-1}),\{0\})}.
\end{align*}
If $n\geq m_{k-1}+2$, then by applying (C3) to $n-1$, we have that
\[
\sigma^{(k)}_{\FS_{(S,+)}(a_{m_{k-1}},\dots,a_{n-2})}\left(U_{N^{(k)}_{n-1}}\right)\subseteq U_{(\FS_{(S,+)}(a_{m_0},\dots,a_{m_1-1}),\dots,\FS_{(S,+)}(a_{m_{k-2}},\dots,a_{m_{k-1}-1}),\{0\})}.
\]
Hence, by the definition of $N_{n}^{(k)}$, we have that 
\begin{align*}
\sigma^{(k)}_{\FS_{(S,+)}(a_{m_{k-1}},\dots,a_{n-2})+a_{n-1}}\left(U_{N^{(k)}_{n}}\right)&\subseteq\sigma^{(k)}_{\FS_{(S,+)}(a_{k-1},\dots,a_{n-2})}\left(U_{N^{(k)}_{n-1}}\right)\\
&\subseteq U_{(\FS_{(S,+)}(a_{m_0},\dots,a_{m_1-1}),\dots,\FS_{(S,+)}(a_{m_{k-2}},\dots,a_{m_{k-1}-1}),\{0\})}.
\end{align*}
To sum up,
\begin{align*}
&\sigma^{(1)}_{\FS_{(S,+)}(a_{m_{k-1}},\dots,a_{n-1})}\left(U_{N^{(k)}_{n}}\right)
\\ = &
\sigma^{(k)}_{\FS_{(S,+)}(a_{m_{k-1}},\dots,a_{n-2})}\left(U_{N^{(k)}_n}\right)\bigcup\sigma^{(k)}_{\FS_{(S,+)}(a_{m_{k-1}},\dots,a_{n-2})+a_{n-1}}\left(U_{N^{(k)}_n}\right)\bigcup \sigma^{(k)}_{a_{n-1}}\left(U_{N^{(k)}_n}\right)
\\ \subseteq & U_{(\FS_{(S,+)}(a_{m_0},\dots,a_{m_1-1}),\dots,\FS_{(S,+)}(a_{m_{k-2}},\dots,a_{m_{k-1}-1}),\{0\})}.
\end{align*}

When $k=1$, by applying (C2) to $n-1,\ldots,0$ and the choice of $N_{n}^{(1)}$, we know that 
$$\sigma^{(1)}_{\FS_{(S,+)}(a_{m_{0}},\dots,a_{n-1})}\left(U_{N^{(1)}_{n}}\right)\subseteq U_{N_{m_0}^{(1)}}\subseteq U_{\{0\}}.$$

Therefore, for $N_{n}^{(i)},1\le i\le \min(d,n)$, (C3) holds.

By the construction of sequence $\{a_n\}_{n\ge 0}\subseteq S$, we know that for any $(m_0,\ldots,m_d)\in \Z^{d+1}$ with $0\leq m_0<m_1<\dots<m_d$,
\begin{align*}
&\xi^{(d)}(\FS_{(S,+)}(a_{m_0},\dots,a_{m_1-1}),\dots,\FS_{(S,+)}(a_{m_{d-1}},\dots,a_{m_d-1}))\\
=&\left(\sigma^{(d)}_{\FS_{(S,+)}(a_{m_{d-1}},\dots,a_{m_d-1})}\xi^{(d)}\right)(\FS_{(S,+)}(a_{m_0},\dots,a_{m_1-1}),\dots,\\ & \hspace{6cm}\FS_{(S,+)}(a_{m_{d-2}},\dots,a_{m_{d-1}-1}),0)\\
=&\left(p\xi^{(d)}\right)(\FS_{(S,+)}(a_{m_0},\dots,a_{m_1-1}),\dots,\FS_{(S,+)}(a_{m_{d-2}},\dots,a_{m_{d-1}-1}),0)\\
=&\xi^{(d-1)}(\FS_{(S,+)}(a_{m_0},\dots,a_{m_1-1}),\dots,\FS_{(S,+)}(a_{m_{d-2}},\dots,a_{m_{d-1}-1}))\\
=&\cdots\\
=&\xi^{(1)}(\FS_{(S,+)}(a_{m_0},\dots,a_{m_1-1}))\\
=&\left(\sigma^{(1)}_{\FS_{(S,+)}(a_{m_0},\dots,a_{m_1-1})}\xi^{(1)}\right)(0)\\
=&\left(p\xi^{(1)}\right)(0).
\end{align*}

Now, we have proved that for any finite coloring of $S_{0}^d$, there is a sequence $\{a_i\}_{i=0}^{\infty}$ in $S$ such that the set
\[
\bigcup_{(m_0,\ldots,m_d)\in \Z^{d+1},\atop 0\leq m_0<m_1<\dots<m_d}\{(x_1,\dots,x_d):\ x_i\in\FS_{(S,+)}(\{a_j\}_{j=m_{i-1}}^{m_{i}-1})\ \text{for each}\ 1\leq i\leq d\}
\]
is monochromatic.

Note that every finite coloring of $S^d$ can be extended to be a finite coloring of $S_{0}^{d}$ and for any $0\le i\le j$, $\FS_{(S,+)}(a_i,\ldots,a_j)\subset S$.
Then this completes the proof.
\end{proof}

\section{The second proof of Theorem \ref{higher}: Ultrafilter method}\label{sec4}

In this section, we give a proof for Theorem \ref{higher} by using the ultrafilter approach. As Furstenberg's comment \cite[First paragraph of Section 3, Chapter 8]{F} on the dynamical proof and the ultrafilter proof for Hindman's theorem, it is similar here. 

Before the proof, we introduce some notions. For an ultrafilter $p\in \beta S$, we define its the \textbf{$k$-th tensor} $p^k$ by the following:
\[
A\in p^k\iff\{s_1\in S:\ \{s_2\in S:\{\cdots\{s_k\in S:\ (s_1,\ldots,s_k)\in A\}\in p\}\cdots\}\in p\}\in p.
\]
Clearly, $p^n\in\beta\left(S^n\right)$ for any $n\ge 1$.

Here, we give a remark for the definition. For the usual embedding function from $S^n$ to $\beta\left(S^n\right)$, one can lift it to be a map from $(\beta S)^n$ to $\beta\left(S^n\right)$. At this time, $p^n$ is exactly the image of $(p,p,\dots,p)\in \left(\beta S\right)^n$ under the map. For more on the tensors of ultrafilters, see \cite[Section 1, Chapter 11]{HS}.

Now, let us recall some facts for the idempotents in $(\beta S,+)$. Recall that
\[
A\in p+q\quad\text{if and only if}\quad\{s\in S: A-s\in q\}\in p.
\]
Hence, when $p$ is an idempotent in $(\beta S,+)$ and $A\in p$, there is
\begin{equation}\label{idem}
\{s\in S:\ A-s\in p\}\in p.
\end{equation}
\eqref{idem} is the core of the proof presented below, and will be frequently used.

Note that for any finite coloring of $S^d$ and any $p\in \beta S$, there is one color class that belongs to $p^d$. Therefore, the following implies Theorem \ref{higher}.

\begin{thm}
Let $d\in \N, (S,+)$ be a semigroup and $p$ be an idempotent in $(\beta S,+)$. Then for all $C\in p^d$, there is a sequence $\{a_i\}_{i=0}^{\infty}$ in $S$ such that for all $(m_0,\ldots,m_d)\in \Z^{d+1}$ with $0\le m_0<m_1<\dots<m_d$, one has
\[
\{(x_1,\dots,x_d):\ x_i\in\FS_{(S,+)}(a_{m_{i-1}},\dots,a_{{m_i}-1}),1\leq i\leq d\}\subseteq C.
\]
\end{thm}

\begin{proof}
Firstly, for the case $d=1$, we inductively construct three sequences $\{b_n\}_{n\ge 0}\subseteq S$, $\{B_n\}_{n\ge 0}\subseteq p$ and $\{D_n\}_{n\ge 0}\subseteq p$. Let $B_0=C,D_0=B_0\cap \{s\in S:B_0-s\in p\}$. By\eqref{idem}, $D_0\in p$. Take $b_0$ from $D_0$ arbitrarily.
Suppose that we have constructed $\{b_j\}_{0\le j<n}\subseteq S$, $\{B_j\}_{0\le j<n}\subseteq p$ and $\{D_j\}_{0\le j<n}\subseteq p$, where $n\ge 1$. Next, we construct $b_n$, $B_n$ and $D_n$. Let $B_n=D_{n-1}\cap (B_{n-1}-a_{n-1})\in p$ and $D_n=B_n\cap \{s\in S:B_n-s\in p\}\in p$. Take $b_n$ from $D_n$ arbitrarily. It is easy to see that for any $0\le m_0<m_1$, we have $\FS_{(S,+)}(b_{m_0},\ldots,b_{m_{1}-1})\subseteq C$.

For the rest of the proof, we can assume that $d>1$. Let
\[
C^{(1)}=\{s_1\in S:\ \{s_2\in S:\{\cdots\{s_d\in S:\ (s_1,\dots,s_d)\in C\}\in p\}\cdots\}\in p\}.
\] For any $2\leq k\leq d$ and any $s_1,\dots,s_{k-1}\in S$, let
\[
C^{(k)}_{s_1,\dots,s_{k-1}}=\{s_k\in S:\ \{s_{k+1}\in S:\{\cdots\{s_d\in S:\ (s_1,\dots,s_d)\in C\}\in p\}\cdots\}\in p\}.
\]
Then $C^{(1)}\in p$, and
\begin{equation}\label{u1}
	C^{(k)}_{s_1,\dots,s_{k-1}}\in p\quad\text{if}\quad s_i\in C^{(i)}_{s_1,\dots,s_{i-1}} \quad\text{for each}\  2\le i<k\ \text{and}\ s_1\in C^{(1)}.
\end{equation}
Next, we inductively construct three sequences $\{a_n\}_{n\ge 0}\subseteq S$, $\{C_n\}_{n\ge 0}\subseteq p$ and $\{A_n\}_{n\ge 0}\subseteq p$ under the following requirements:
\begin{itemize}
	\item[(C1)] for any $n\ge 0$, $a_n\in C_n=A_n\cap \{s\in S:A_n-s\in p\}$;
	\item[(C2)] for any $n\ge 0$, any $2\le j\le \min(d,n+2)$, any $(m_0,\ldots,m_{j-1})\in \Z^j$ with $0\le m_0<\cdots<m_{j-1}<n+2$ and any $x_i\in \FS_{(S,+)}(a_{m_{i-1}},\ldots,a_{m_{i}-1}),1\le i\le j-1$, one has
	\begin{equation}\label{c2}
		C^{(j)}_{x_1,\ldots,x_{j-1}}\in p\ \text{and}\ x_{j-1}\in C^{(j-1)}_{x_1,\ldots,x_{j-2}},
	\end{equation}
	 where we view $C^{(1)}_{x_1,\ldots,x_0}$ as $C^{(1)}$.
\end{itemize}
By \eqref{idem}, $\{s\in S:C^{(1)}-s\in p\}\in p$. Let $$C_0=C^{(1)}\cap \{s\in S:C^{(1)}-s\in p\}\in p,A_0=C^{(1)}\in p.$$ Take $a_0$ from $C_0$ arbitrarily. Then for $n=0$, (C1) and (C2) hold.
Suppose that we have constructed $\{a_j\}_{0\le j<n}\subseteq S$, $\{C_j\}_{0\le j<n}\subseteq p$ and $\{A_j\}_{0\le j<n}\subseteq p$ under (C1) and (C2), where $n\ge 1$. Next, we construct $a_n$, $C_n$ and $A_n$.

Let $A_n$ be the set $$C_{n-1}\bigcap \left(A_{n-1}-a_{n-1}\right)\bigcap \left(\bigcap_{j=2}^{\min(d,n+1)}\bigcap_{(m_0,\ldots,m_{j-1})\in \Z^j,\atop 0\le m_0<\cdots<m_{j-1}<n+1}\bigcap_{x_i\in \FS_{(S,+)}(a_{m_{i-1}},\ldots,a_{m_{i}-1}),\atop 1\le i\le j-1}C^{(j)}_{x_1,\ldots,x_{j-1}}\right)$$ and $$C_n=A_n\cap \{s\in S:A_n-s\in p\}.$$ Then by the inductive hypothesis and \eqref{idem}, $A_n,C_n\in p$.
Take $a_n$ from $C_n$ arbitrarily. Clearly, for $n$, (C1) holds. Next, we verify that for $n$, (C2) holds.

Fix $2\le j\le \min(d,n+2)$, $(m_0,\ldots,m_{j-1})\in \Z^j$ with $0\le m_0<\cdots<m_{j-1}<n+2$ and $$x_i\in \FS_{(S,+)}(a_{m_{i-1}},\ldots,a_{m_{i}-1}),1\le i\le j-1.$$ Let us consider the case $d\le n+1$ first.
If $m_{j-1}<n+1$, then by applying (C2) to $n-1$, we know that \eqref{c2} holds.
If $m_{j-1}=n+1$, when $m_{j-2}=n$, $x_{j-1}=a_n$;
when $m_{j-2}<n$, $x_{j-1}=y$ or $y+a_n$, where $y\in \FS_{(S,+)}(a_{m_{j-2}},\ldots,a_{n-1})$. By the construction of $C_n$, one has $$a_n\in C_n\subseteq C^{(j-1)}_{x_1,\ldots,x_{j-2}}(j>2), a_n\in C_n\subseteq C^{(1)},$$ $$a_n\in C_n\subseteq C^{(j-1)}_{x_1,\ldots,x_{j-2}}-y(j>2)\ \text{and}\ a_n\in C_n\subseteq C^{(1)}-y.$$ So, $x_{j-1}\in C^{(j-1)}_{x_1,\ldots,x_{j-2}}$. By applying (C2) to $n-1$ and \eqref{u1}, we know that $$C^{(j)}_{x_1,\ldots,x_{j-1}}\in p.$$ 

Now, let us consider the case $d>n+1$. If $2\le j\le n+1$, by a similar argument as one of the case $d\le n+1$, we know that \eqref{c2} holds. If $j=n+2$, we need to verify that $$C^{(n+2)}_{a_0,\ldots,a_n}\in p\ \text{and}\ a_{n}\in C^{(n+1)}_{a_0,\ldots,a_{n-1}}.$$
By the construction of $C_n$, we have $a_n\in C_n\subseteq A_n\subseteq C^{(n+1)}_{a_0,\ldots,a_{n-1}}$. Based on this, by applying (C2) to $n-1$ and \eqref{u1}, we have $C^{(n+2)}_{a_0,\ldots,a_n}\in p$.

To sum up, for $n$, (C2) holds.

Now, we verify that the sequence $\{a_n\}_{n\ge 0}$ satisfies our requirement. Fix $(m_0,\ldots,m_d)\in \Z^{d+1}$ with $0\le m_0<m_1<\dots<m_d$ and $x_i\in\FS_{(S,+)}(a_{m_{i-1}},\dots,a_{{m_i}-1}),1\leq i\leq d$. Assume that $x_d=a_r+\cdots+a_n,$ where $m_{d-1}\le r\le n\le {m_d}-1$. Note that if $r<n$, we let $x_d=z+a_n$. Then by the construction of $C_n$, $$a_n\in (A_{n-1}-a_{n-1})\cap A_{n-1}\subseteq\cdots \subseteq A_{r}-z\subseteq C^{(d)}_{x_1,\ldots,x_{d-1}}-z.$$ If $r=n$, then $x_d=a_n\in C_n\subseteq A_n\subseteq C^{(d)}_{x_1,\ldots,x_{d-1}}$.

So, $\displaystyle x_d\in C^{(d)}_{x_1,\ldots,x_{d-1}}$. Then by (C2) and \eqref{u1}, we have $(x_1,\ldots,x_d)\in C$.
This completes the proof.
\end{proof}

Now, we briefly illustrate the relation between dynamical proof and ultrafilter proof. To be more understandable, let us focus on $(\N,+)$. Let $\xi$ be a finite coloring $\N_0\to \{1,\ldots,c\}$, and $p$ be an idempotent in $(\beta \N,+)$. For each $1\le i\le c$, let $C_i=\N\cap \xi^{-1}(\{i\})$. Then
\[
\left(p\xi\right)(0)=\text{the unique color}\ i_0\ \text{such that }\ C_{i_0}\in p.
\]
Generally, for each $s\in S$,
\[
\left(p\xi\right)(s)=\text{the unique color}\ i_s\ \text{such that}\ C_{i_s}-s\in p.
\]
In the dynamical proof, we need that $\xi$ and $p\xi$ are recurrent to $p\xi$ simultaneously.
For a neighborhood
\[
U_N:=\{x\in \{1,\dots,c\}^{\N_0}:\ x(t)=(p\xi)(t),\forall t\in N\},
\] of $p\xi$, 
where $N$ is a finite subset of $\N_0$, we have
\[
\sigma_s\xi\in U_N \iff s\in\bigcap_{t\in N}(C_{i_t}-t).
\]
We also need $\sigma_s(p\xi)\in U_N$, which means that
\[
(p\xi)(t+s)=(p\xi)(t)=i_t\quad\text{for each}\  t\in N,
\]
Therefore, we need that $C_{i_t}-t-s\in p$ for any $t\in N$, which is equivalent to
\[
s\in\bigcap_{t\in N}\{r\in \N:\ C_{i_t}-t-r\in p\}.
\]
This explains the reason why the dynamical proof and the ultrafilter proof run through similar lines.

\section{Proofs of Theorem \ref{multi} and \ref{multi2}}\label{sec5}

 Note that by any one of two proofs for Theorem \ref{higher}, we actually prove the following.

\begin{thm}\label{higher2}
Let $(S,+)$ be a semigroup, $d\in\N$ and $p$ be an idempotent in $(\beta S,+)$. Then for any finite coloring of $S^d$, there is $C_0\in p$ such that for any $a_0\in C_0$, there is $C_{1;a_0}\in p$ such that for any $a_1\in C_{1;a_0}$, there is $C_{2;a_0,a_1}\in p$ such that ...... 
finally, we get a sequence $\{a_n\}_{n\ge 0}$ in $S$ such that
the set
\[
\bigcup_{(m_0,\ldots,m_d)\in \Z^{d+1},\atop 0\leq m_0<m_1<\dots<m_d}\{(x_1,\dots,x_d):\ x_i\in\FS_{(S,+)}(a_{m_{i-1}},\dots,a_{{m_i}-1})\ \text{for each}\  1\leq i\leq d\}
\]
is monochromatic and the related color only depends on $p$ and correspoding finite coloring of $S^d$.
\end{thm}
Now, we begin to prove Theorem \ref{multi}.
\begin{proof}[Proof of Theorem \ref{multi}]
Fix $c$-coloring $\xi\in \{1,\dots,c\}^S$ and $1\le d\le D$. We color $S^d$ such that
\[
(s_1,s_2,\dots,s_d)\ \text{has color}\ \xi\left((((s_1\circ s_2)\circ s_3)\circ\cdots)\circ s_d\right)\quad \text{for each}\  (s_1,\ldots,s_d)\in S^d.
\]And we use the symbol $\xi^{(d)}$ to denote this color.
Take an idempotent $p$ in $(\beta S,+)$ and fix it. Then by applying Theorem \ref{higher2} to $\xi^{(d)}$ and $p$, there is $C_0^{(d)}\in p$ such that for any $b_0^{(d)}\in C_0^{(d)}$, there is $C_{1;b_0^{(d)}}^{(d)}\in p$ such that for any $b_1^{(d)}\in C_{1;b_0^{(d)}}^{d}$, there is $C_{2;b_0^{(d)},b_1^{(d)}}^{(d)}\in p$ such that ...... 
finally, we get a sequence $\{b_n^{(d)}\}_{n\ge 0}$ in $S$ such that
the set
\[
\bigcup_{(m_0,\ldots,m_d)\in \Z^{d+1},\atop 0\leq m_0<m_1<\dots<m_d}\{(((x_1\circ x_2)\circ x_3)\circ\cdots)\circ x_d:\ x_i\in\FS_{(S,+)}(b_{m_{i-1}}^{(d)},\dots,b_{{m_i}-1}^{(d)})\ \text{for each}\  1\leq i\leq d\}
\]
is monochromatic (with respect to $\xi$) and the related color only depends on $d$.

Let $a_0\in \displaystyle\bigcap_{d=1}^{D}C^{(d)}_0$. Suppose that we have constructed $\{a_m\}_{0\le m<n}$, where $n\ge 1$. Now, let $$a_n\in \displaystyle\bigcap_{d= 1}^{D}C^{(d)}_{n;a_0,\ldots,a_{n-1}}.$$ 

Then we get a sequence $\{a_n\}_{n\ge 0}$ in $S$ which satisfies the property that for all $1\le d\le D$, the set
\[
\bigcup_{(m_0,\ldots,m_d)\in \Z^{d+1},\atop 0\leq m_0<m_1<\dots<m_d}\{(((x_1\circ x_2)\circ x_3)\circ\cdots)\circ x_d:\ x_i\in\FS_{(S,+)}(\{a_j\}_{j=m_{i-1}}^{m_{i}-1})\ \text{for each}\ 1\leq i\leq d\}
\]
is monochromatic (with respect to $\xi$), where the color of the set depends on $d$. This finishes the proof.
\end{proof}
Next, we are about to prove Theorem \ref{multi2}.
\begin{proof}[Proof of Theorem \ref{multi2}]
Fix a $c$-coloring $\xi\in \{1,\dots,c\}^{\N}$. Take an idempotent $p$ from $(\beta \N,+)$ and fix it. As in the proof of Theorem \ref{multi}, for any $d\in \N$, there is $C_0^{(d)}\in p$ such that for any $b_0^{(d)}\in C_0^{(d)}$, there is $C_{1;b_0^{(d)}}^{(d)}\in p$ such that for any $b_1^{(d)}\in C_{1;b_0^{(d)}}^{d}$, there is $C_{2;b_0^{(d)},b_1^{(d)}}^{(d)}\in p$ such that ...... 
finally, for any $d\in\N$, we get a sequence $\{b_n^{(d)}\}_{n\ge 0}$ in $S$ such that
the set
\[
\bigcup_{(m_0,\ldots,m_d)\in \Z^{d+1},\atop 0\leq m_0<m_1<\dots<m_d}\{(((x_1\circ x_2)\circ x_3)\circ\cdots)\circ x_d:\ x_i\in\FS_{(\N,+)}(b_{m_{i-1}}^{(d)},\dots,b_{{m_i}-1}^{(d)}),\forall 1\leq i\leq d\}
\]
is monochromatic (with respect to $\xi$) and the related color only depends on $d$,  where we use $i_d$ to denote the color.

So, we get a finite coloring $\eta:\N\to \{1,\ldots,c\}$ such that for each $d\in \N$, $\eta(d)=i_d$.

Applying Theorem \ref{higher2} to $(\N,+),\eta$ and $p$, there is $D_0\in p$ such that for any $x_0\in D_0$, there is $D_{1;x_0}\in p$ such that for any $x_1\in D_{1;x_0}$, there is $D_{2;x_0,x_1}\in p$ such that ...... 
finally, we get a sequence $\{x_n\}_{n\ge 0}$ in $S$ such that
the set $\FS_{(\N,+)}(\{x_n\}_{n\ge 0})$ is monochromatic (with respect to $\eta$).

Take $a_0$ from $D_0$ arbitrarily. Suppose that we have constructed $\{a_j\}_{0\le j<n}$, where $n\ge 1$. Then we take $a_n$ from $\displaystyle D_{n;a_0,\ldots,a_{n-1}}\bigcap \left(\bigcap_{j=1}^{n}\bigcap_{d\in \FS_{(\N,+)}(a_0,\ldots,a_{j-1})}C^{(d)}_{n-j;a_j,\ldots,a_{n-1}}\right)$ arbitrarily, where for each $d\in \N$, we view $C^{(d)}_{0;a_n,\ldots,a_{n-1}}$ as $C^{(d)}_0$.

By the construction of $\{a_n\}_{n\ge 0}$, we know that the set
\begin{align*}
\bigcup_{m_0\in\N}\bigcup_{d\in\FS_{(\N,+)}(\{a_j\}_{j=0}^{m_{0}-1})}\bigcup_{(m_1,\ldots,m_d)\in \N^{d},\atop m_0<m_1<\dots<m_d}&\{(((x_1\circ x_2)\circ x_3)\circ\cdots)\circ x_d:\\
&x_i\in\FS_{(\N,+)}(\{a_j\}_{j=m_{i-1}}^{m_{i}-1})\ \text{for each}\  1\leq i\leq d\}
\end{align*}
is monochromatic (with respect to $\xi$). This finishes the proof. 
\end{proof}

\bibliographystyle{plain}
\bibliography{ref}
\end{document}